\documentclass[11pt,reqno]{amsart}

\usepackage[margin=1.1in]{geometry}
\usepackage{amsmath,amssymb,amsthm,mathtools}
\usepackage{enumitem}
\usepackage[dvipsnames]{xcolor}
\usepackage{tikz}
\usepackage{float}
\usepackage[colorlinks=true,linkcolor=blue!60!black,citecolor=red!60!black,urlcolor=blue!60!black,pagebackref=true]{hyperref}
\numberwithin{equation}{section}
\newtheorem{theorem}{Theorem}[section]
\newtheorem{lemma}[theorem]{Lemma}
\newtheorem{proposition}[theorem]{Proposition}

\theoremstyle{definition}
\newtheorem{definition}[theorem]{Definition}
\theoremstyle{remark}
\newtheorem{remark}[theorem]{Remark}

\newcommand{\R}{\mathbb R}
\newcommand{\dd}{\,\mathrm d}
\newcommand{\TM}{\mathrm{TM}}
\newcommand{\eps}{\varepsilon}
\newcommand{\weakto}{\rightharpoonup}
\newcommand{\dist}{\operatorname{dist}}

\begin{document}
\title[Extremals via polygonal profiles]{Extremals for Trudinger--Moser inequalities via polygonal profiles}
\author{Jos\'e Francisco de Oliveira}
\address{Department of Mathematics, Federal University of Piau\'i, 64049-550 Teresina, PI, Brazil}
\email{jfoliveira@ufpi.edu.br}
\begin{abstract}
  The aim of this paper is to establish the existence of extremal functions for the Trudinger--Moser inequality with logarithmic weights $(-\log |x|)^{\beta}$ in two dimensions throughout the full range $0\le \beta<1$. We also extend recent results on the existence of local maximizers for the Trudinger--Moser functional beyond the critical threshold. Our approach uses polygonal profiles to establish a lower bound uniform in the weight parameter. An explicit thirty-node profile gives the lower bound $\mathrm{TM}_{4\pi}>4.38421793217\,\pi$ for the classical critical supremum on the unit disk, improving on the value $4.3556\,\pi$ obtained in the computer experiments of Carleson and Chang. This lower bound has a purely analytic proof, which uses only elementary inequalities and exact rational arithmetic. Numerical experiments with polygonal profiles of up to $102400$ nodes suggest that this lower bound can be raised to $4.402732044\,\pi$, which improves the recent computational lower bound by Li and Sun.
\end{abstract}
\maketitle

%======================================================================
\section{Introduction and main results}\label{sec:introduction}
%====================================================================
Let $B\subset\R^2$ be the open unit disk and denote by $H^1_0(B)$ the Sobolev space endowed with the Dirichlet norm. The classical Trudinger--Moser inequality \cite{Trudinger67,HempelMorrisTrudinger70,Moser71} identifies exponential growth as optimal for the embedding of $H^1_0(B)$ into the class of Orlicz spaces. In particular, the well-known Trudinger--Moser functional $J_{\alpha}:H^1_0(B)\to \mathbb{R}$ given by
\begin{equation}
    J_{\alpha}(u)=\int_{B}e^{\alpha u^{2}}\dd x, \;\; \alpha\ge 0
\end{equation}
is well-defined and satisfies
\begin{equation}\label{Moser-sup}
   \mathrm{TM}_{\alpha}=\sup_{\substack{u\in H^1_0(B),\;
   \|\nabla u\|_2\le 1}}
 J_{\alpha}(u) \quad\left\{\begin{aligned}
        &<\infty \;&\mbox{if}&\;\; \alpha\le 4\pi\\
        &=\infty \;&\mbox{if}&\;\; \alpha>4\pi,
    \end{aligned}\right.
\end{equation}
where $\|\cdot\|_2$ denotes the standard norm in the Lebesgue space $L^2(B)$. Especially at the critical threshold $\alpha=4\pi$, the existence of a maximizer for the supremum \eqref{Moser-sup} is delicate, due to the influence of optimal exponential growth that causes a lack of compactness; see, for instance, \cite{Lions85,CernyCianchiHencl13}. In spite of that, Carleson and Chang \cite{CarlesonChang86} were able to show that the supremum admits a maximizer in the critical case, by an approach currently known as the two-step strategy of Carleson--Chang. For a nonsymmetric domain $\Omega\subset\mathbb{R}^N$ with general $N\ge 2$, the supremum \eqref{Moser-sup} can also be reformulated \cite{Moser71} and its maximizing problem was investigated by Struwe \cite{Struwe88}, Flucher \cite{Flucher92} and Lin \cite{Lin96}. 
This kind of estimate has a close relation with geometric analysis \cite{ChangYang03,Onofri82}
and PDEs \cite{Adimurthi90,deFigueiredodoORuf02} and a lot of generalizations can be found:
for unbounded domains \cite{Ruf05,LiRuf08,doO97,Cao92},
for sharpened forms \cite{AdimurthiDruet04,Tintarev14},
for Riemannian manifolds \cite{Fontana93,Li05}, and for weighted measures
\cite{AdimurthiSandeep07,AdimurthiYang10,OliveiraDoO14}.
Here we are interested in a Trudinger--Moser inequality in which the Dirichlet norm constraint
is subjected to a logarithmic weight, obtained by Calanchi and Ruf \cite{CalanchiRuf15}. Precisely, let $H^1_{0,\mathrm{rad}}(B,w_\beta)$  be  the completion of $C^\infty_{0,\mathrm{rad}}(B)$ with respect to the weighted norm
\begin{equation}\label{eq:norm}
\|u\|_{w_\beta}^2=\int_B|\nabla u|^2w_\beta\dd x,\;\; \mbox{where}\;\; w_\beta(x)=\left(-\log|x|\right)^\beta,\;\; 0\le \beta<1.
\end{equation}
Then, the functional $J_{\alpha,\beta}:H^1_{0,\mathrm{rad}}(B,w_\beta)\to\mathbb{R}$ given by 
\begin{equation}
  J_{\alpha,\beta}(u)=\int_{B}e^{\alpha |u|^{\frac{2}{1-\beta}}}\dd x
\end{equation}
is well-defined and 
\begin{equation}\label{CR-Moser-sup}
   \mathrm{TM}_{\alpha,\beta}=\sup_{\substack{u\in H^1_{0,\mathrm{rad}}(B,w_\beta),\;
\|u\|_{w_\beta}\le1}}
 J_{\alpha,\beta}(u)<\infty \;\;\mbox{iff}\;\; \alpha\le \alpha_\beta:=2[2\pi(1-\beta)]^{\frac{1}{1-\beta}}.
\end{equation}
The Calanchi--Ruf inequality \eqref{CR-Moser-sup} has recently attracted the attention of several authors.
Extensions to $N$-dimensional Euclidean spaces can be found in \cite{CalanchiRufN15},
to unbounded domains in \cite{AouaouiJlel20,Aouaoui25},
to Sobolev spaces with fractional dimensions in \cite{XueZhangZhu25},
and to the $k$-Hessian operator in \cite{DoOOliveiraPonciano26}.
Despite the many studies of this inequality, the central question of the existence of extremals
in the critical case $\alpha=\alpha_\beta$ throughout the full range $0\le\beta<1$
remains open, to the best of our knowledge.
Indeed, Roy \cite{Roy16} claimed to have answered this question positively;
however, Nguyen \cite{Nguyen19} identified a gap in Roy's argument, as also
acknowledged by Roy himself in \cite{RoyN19}. Nevertheless, Roy's method remains
valid for sufficiently small $\beta>0$ and therefore yields the existence of
extremals for \eqref{CR-Moser-sup} in this range. More precisely, the gap lies
only in the test function used in the proof of \cite[Theorem~3]{Roy16}, for which
the inequality $\|u\|_{w_\beta}\le 1$ is not guaranteed to hold for all
$\beta\in[0,1)$; after a suitable normalization, this test function yields the
desired conclusion for $\beta>0$ small enough (see \cite{RoyN19} for details). Our first result answers this question. Its proof uses normalized test functions
built from polygonal profiles (see Lemma~\ref{lem:test2} and Proposition~\ref{prop:explicit}).
\begin{theorem}\label{thm:attainment}
For every $0\le\beta<1$, we have that $\mathrm{TM}_{\alpha_{\beta},\beta}$ is attained. 
\end{theorem}
To clarify our main contribution, we outline below the approach used in the proof of Theorem~\ref{thm:attainment}. Inspired by the two-step strategy of Carleson--Chang, the idea is to exclude the concentration phenomenon for any normalized maximizing sequence. Following Nguyen \cite{Nguyen19} (see also \cite{Roy16}), we say a sequence $(u_k)\subset H^1_{0,\mathrm{rad}}(B,w_\beta)$ is a normalized concentrating sequence (NCS for short) if
\begin{equation}\label{eq:concentration-definition}
\|u_k\|_{w_\beta}=1,\;\; u_k\weakto0\text{ in }H^1_{0,\mathrm{rad}}(B,w_\beta)\;\;\mbox{and}\;\;
\int_{B\setminus B_\rho}|\nabla u_k|^2w_\beta\dd x\rightarrow0,
\;\;\mbox{for any}\;\;\rho\in(0,1)
\end{equation}
where $B_\rho$ denotes the disk of radius $\rho$ centered at the origin, and the (critical) concentration level is defined by
\[
\Sigma_{\beta}=\sup\left\{\limsup_{k\to\infty}J_{\alpha_{\beta}, \beta}(u_k):
(u_k)\text{ is NCS}\right\}.
\]
Roy \cite[Lemma~9]{Roy16} established the concentration bound (see also \cite[(1.4)]{Nguyen19})
\begin{equation}\label{concentration-B}
  \Sigma_\beta\le\pi(1+e),\;\; \mbox{for all}\;\; 0\le \beta<1;
\end{equation}
for $\beta=0$, this is the Carleson--Chang concentration level $\Sigma_0=\pi(1+e)$ \cite{CarlesonChang86}.
On the other hand, Nguyen \cite[Theorem~1.1 and Theorem~1.2]{Nguyen19} proved the monotonicity 
\begin{equation}\label{mono-F}
 \mathrm{TM}_{\alpha_{\beta^{\prime}},\beta'}\le\mathrm{TM}_{\alpha_\beta,\beta} \;\;\mbox{and}\;\;\Sigma_{\beta^{\prime}}\le \Sigma_{\beta}\,\;\;\mbox{for}\;\;0\le\beta\le\beta'<1.
\end{equation}
Set $\mathcal S_\beta=\{u\in H^1_{0,\mathrm{rad}}(B,w_\beta)\;:\; \|u\|_{w_\beta}=1\}$.
Let $(u_k)\subset\mathcal{S}_{\beta}$ be a maximizing sequence, that is, $J_{\alpha_\beta,\beta}(u_k)\to \mathrm{TM}_{\alpha_\beta,\beta}$. Up to a subsequence, we may assume that $u_k\weakto u$ in $H^1_{0,\mathrm{rad}}(B,w_\beta)$, and according to \cite[Lemma~5]{Roy16}, either $J_{\alpha_\beta,\beta}(u_k)\to J_{\alpha_\beta,\beta}(u)$, meaning $u$ is an extremal function, or $(u_k)$ concentrates at the origin, in which case $u\equiv0$ and $(u_k)$ is an NCS. In particular, to ensure the existence of a maximizer, we only need to show that $\mathrm{TM}_{\alpha_\beta,\beta}$ is strictly above the concentration level $\Sigma_{\beta}$. To get this, we introduce a class of $n$-node polygonal profiles (cf. Definition~\ref{nodes-define}) for which, with the help of \eqref{mono-F}, we prove a uniform lower bound (cf. Proposition~\ref{prop:uniform-profile}). Then, we exhibit an explicit three-node polygonal profile (cf. Proposition~\ref{prop:explicit}) that ensures
\begin{equation}\nonumber
  \TM_{\alpha_\beta,\beta}>\frac{2267}{600}\pi>\pi(1+e)\ge \Sigma_{\beta},\;\;\mbox{for all}\;\; 0\le \beta<1.
\end{equation}
Our second aim concerns the existence of local maximizers for $J_{\alpha,\beta}$ beyond the critical threshold $\alpha_\beta$. Although the global supremum is then infinite, local maximizers may persist near the critical extremals. Struwe \cite{Struwe88} proved this phenomenon for the classical functional $\beta=0$. For a more careful analysis, we recommend \cite{delPinoMussoRuf12,MalchiodiMartinazzi14} and the references therein. Recently,  Da Silva Paiva and de Oliveira \cite{PaivaOliveira26} established the counterpart of Struwe's result with logarithmic weights for sufficiently small positive $\beta$. Here, with the help of Theorem~\ref{thm:attainment} we are able to extend their argument to the full parameter range $0\le\beta<1$.
\begin{theorem}\label{thm:supercritical}
For every $0\le\beta<1$ there exists $\alpha^{\sharp}>\alpha_{\beta}$ such that $J_{\alpha,\beta}|_{\mathcal{S}_\beta}$ admits a local maximizer for any $0\le\alpha<\alpha^{\sharp}$.
\end{theorem}
The proof of Theorem~\ref{thm:supercritical} uses an approach close to the one in \cite{Struwe88}, which has already been successfully applied in \cite{PaivaOliveira26} for $\beta$ small enough and in \cite{deOliveiradeSouzaMedeiros} for the singular Trudinger--Moser functional.

The third aim is to obtain a quantitative estimate for the classical Trudinger--Moser supremum on the unit disk. In the computer experiments reported by Carleson and Chang \cite[Section~4]{CarlesonChang86}, the largest value of the functional found was approximately $4.3556\,\pi$, and explicit upper bounds were obtained by Kim \cite{Kim00} and, more recently, by Li and Sun \cite{LiSun26} (see Remark~\ref{rem:numerical-classical}). An explicit thirty-node profile and Jensen's inequality suffice to improve on the value of Carleson and Chang.
\begin{theorem}\label{thm:classical-bounds}
The Trudinger--Moser supremum \eqref{Moser-sup} satisfies $\mathrm{TM}_{4\pi}>4.38421793217\,\pi$.
\end{theorem}
The lower bound $4.38421793217\,\pi$ follows from a direct estimate that uses only elementary
inequalities and exact rational arithmetic.

\begin{remark}\label{rem:numerical-classical}
Kim \cite[Theorem~3.5]{Kim00} proved an upper bound for $\mathrm{TM}_{4\pi}$ with
explicit constants; with the constants given in \cite[Lemma~3.7]{Kim00}, it
yields $\mathrm{TM}_{4\pi}\le2^{4/3}e^{\frac13+\frac{3}{4e}}\,\pi<4.6341\,\pi$
(cf. Remark on p.~451 of \cite{Kim00}).
More recently, Li and Sun proved, by a computer-assisted argument based on
interval arithmetic, that
$4.40273202\,\pi\le\mathrm{TM}_{4\pi}\le4.40273208\,\pi$
(see the proof of \cite[Proposition~7.7]{LiSun26}). The lower bound in
Theorem~\ref{thm:classical-bounds} is weaker, but its proof is purely
analytic. A computer-assisted analysis, based on the numerical optimization of
the spacing and of the heights of polygonal profiles with up to $102400$
uniformly spaced nodes, raises this lower bound to
$4.402732044\,\pi$, which slightly improves the lower bound of Li and Sun. The values obtained in these
experiments are listed in Table~\ref{tab:numerical-classical} in
Appendix~\ref{app:numerical-classical}.
\end{remark}

The paper is organized as follows. Section~\ref{sec:lower} gives the test functions and the uniform lower bound and proves the critical attainment stated in Theorem~\ref{thm:attainment}. In Section~\ref{sec:supercritical}, we extend Lemmas~3--5 of \cite{PaivaOliveira26} to the full range $0\le\beta<1$ and prove Theorem~\ref{thm:supercritical}. Section~\ref{sec:rational-classical} proves Theorem~\ref{thm:classical-bounds} using thirty nodes. Appendix~\ref{app:numerical-classical} lists the values obtained in numerical experiments with finer profiles.

%======================================================================
\section{A direct uniform lower bound}\label{sec:lower}
%======================================================================
In this section, we establish the uniform lower bound
$\mathrm{TM}_{\alpha_\beta,\beta}>\frac{2267}{600}\pi>\pi(1+e)$ for every
$\beta\in[0,1)$. Together with the concentration estimate
\eqref{concentration-B} and the concentration--compactness results in
\cite[Lemma~5]{Roy16} and \cite[Lemma~2.2]{Nguyen19}, this bound rules out concentration of maximizing
sequences and leads to the proof of Theorem~\ref{thm:attainment}.
\begin{definition}\label{nodes-define}
  Let $n\in\mathbb{N}$ and $\phi:[0,\infty)\to[0,\infty)$ be a function. We say that $\phi$ is an $n$-node polygonal profile if there exist  $0=t_0<t_1<\cdots<t_n $ and $0=y_0<y_1\le\cdots\le y_n$
such that $\phi(0)=0$, $\phi(t_j)=y_j$, $\phi$ is affine on $[t_{j-1},t_j]$, and $\phi=y_n$ on $[t_n,\infty)$.
\end{definition}
\begin{lemma}\label{lem:test}
Let $\phi$ be an $n$-node polygonal profile on $[0,\infty) $ and set
\begin{equation}\label{eq:direct-test}
v_{\beta,\phi}(x)=\alpha_\beta^{-\frac{1-\beta}{2}}
\Big[\phi\Big(2\log\frac1{|x|}\Big)\Big]^{1-\beta},\;\;\mbox{for all}\;\;
 x\in B\setminus\{0\}.
\end{equation}
 Then $v_{\beta,\phi}\in H^1_{0,\mathrm{rad}}(B,w_\beta)$ and
\begin{equation}\label{eq:normalization}
\|v_{\beta,\phi}\|_{w_\beta}^2=(1-\beta)\int_0^{t_n}\phi'(t)^2
\left(\frac{t}{\phi(t)^2}\right)^\beta\dd t.
\end{equation}
\end{lemma}
\begin{proof}
Set
\[
\eta(s)=
\Big[\phi\left(s^{\frac{1}{1-\beta}}\right)\Big]^{1-\beta},
\;\;\mbox{for all}\;\; s\ge0.
\]
Then, from the $\phi$ profile, we have that $\eta(s)=a^{1-\beta}s$ on $[0,t_1^{1-\beta}]$ with $a={y_1}/{t_1}>0$, and is constant on $[t_n^{1-\beta},\infty)$. On the remaining intervals, \(\eta\) is piecewise smooth with bounded derivative. Thus \(\eta\) is locally Lipschitz, \(\eta(0)=0\), and $\eta'\in L^2(0,\infty).$ By density, choose $(h_k)\subset C^\infty_0(0,\infty)$ such that \(h_k\to\eta'\) in \(L^2(0,\infty)\), and define $\eta_k(s)=\int_0^s h_k(\tau)\,\mathrm d\tau.$
Now set
$$
v_k(x)=
\alpha_\beta^{-\frac{1-\beta}{2}}
\eta_k\Big(
\Big(2\log\frac{1}{|x|}\Big)^{1-\beta}
\Big),
\;\;\mbox{for}\;\; x\in B\setminus\{0\}.
$$
Because $h_k$ has compact support in $(0,\infty)$, the function $\eta_k$ vanishes near zero and is constant for sufficiently large $s$. So, $v_k$ vanishes near $\partial B$ and is constant near the origin. Extending it by that constant at the origin gives $v_k\in C^\infty_{0,\mathrm{rad}}(B).$
The function $v_{\beta,\phi}$ also extends by a constant at the origin and $v_{\beta,\phi}(x)=0$ if $|x|=1$. The substitution $s=(2\log \frac{1}{r})^{1-\beta}$, together with
$\alpha_\beta^{1-\beta}=4\pi (1-\beta)2^{-\beta}$, gives
\[
\int_B|\nabla(v_k-v_{\beta,\phi})|^2w_\beta\dd x
=\int_0^\infty|\eta_k'(s)-\eta'(s)|^2\dd s
=\|h_k-\eta'\|_{L^2(0,\infty)}^2\rightarrow0,\;\;\mbox{as}\;\; k\to \infty.
\]
Thus, $v_{\beta,\phi}\in H^1_{0,\mathrm{rad}}(B,w_\beta)$. Finally, from the substitution $t=2\log\frac{1}{r}$, we obtain
\begin{equation}\label{change-norm}
\begin{aligned}
\|v_{\beta,\phi}\|_{w_\beta}^2
&=2\pi\int_0^1
|v_{\beta,\phi}'(r)|^2
\left(\log\frac{1}{r}\right)^\beta r\,\mathrm dr\\
&=4\pi(1-\beta)^2
\alpha_\beta^{-(1-\beta)}2^{-\beta}
\int_0^\infty
|\phi'(t)|^2\phi(t)^{-2\beta}t^\beta\,\mathrm dt\\
&=(1-\beta)\int_0^{\infty}
|\phi'(t)|^2
\left(\frac{t}{\phi(t)^2}\right)^\beta\,\mathrm dt
\end{aligned}
\end{equation}
and \eqref{eq:normalization} holds.
\end{proof}
\begin{lemma}\label{lem:test2}
Let $\phi$ and $v_{\beta,\phi}$ be as in Lemma~\ref{lem:test}. Set 
\begin{equation}\label{eq:theta}
\theta_\beta(\phi)=\|v_{\beta,\phi}\|^{\frac{2}{1-\beta}}_{w_{\beta}}
\;\;\mbox{and}\;\;
u_{\beta,\phi}=\frac{v_{\beta,\phi}}{\|v_{\beta,\phi}\|_{w_\beta}}.
\end{equation}
Then 
\begin{equation}\label{eq:direct-lower}
\|u_{\beta,\phi}\|_{w_{\beta}}=1\;\;\;\mbox{and}\;\;\;\mathrm{TM}_{\alpha_\beta,\beta}\ge J_{\alpha_\beta,\beta}(u_{\beta,\phi})=\pi\int_0^\infty
e^{\frac{\phi(t)^2}{\theta_\beta(\phi)}-t}\dd t.
\end{equation}
\end{lemma}
\begin{proof}
By Lemma~\ref{lem:test}, $v_{\beta,\phi}\in
H^1_{0,\mathrm{rad}}(B,w_\beta)$ and
$\|v_{\beta,\phi}\|_{w_\beta}<\infty$.
Since $\phi(t)=(y_1/t_1)t$ on $[0,t_1]$,
\eqref{eq:normalization} gives
\begin{align*}
\|v_{\beta,\phi}\|_{w_\beta}^2
&\ge(1-\beta)\left(\frac{y_1}{t_1}\right)^{2(1-\beta)}
\int_0^{t_1}t^{-\beta}\dd t\\
&=\left(\frac{y_1^2}{t_1}\right)^{1-\beta}>0.
\end{align*}
Thus $u_{\beta,\phi}$ is well defined and
belongs to $H^{1}_{0, \operatorname{rad}}(B, w_{\beta})$ with $\|u_{\beta,\phi}\|_{w_\beta}
=1$. For $0<r<1$, definition \eqref{eq:direct-test} yields
\begin{align*}
\alpha_\beta|u_{\beta,\phi}(r)|^{\frac{2}{1-\beta}}=\frac{\big[\phi\left(2\log\frac1r\right)\big]^2}
{\theta_\beta(\phi)}.
\end{align*}
Then, the change of variable $t=2\log(1/r)$ implies 
\begin{align*}
J_{\alpha_\beta,\beta}(u_{\beta,\phi})
&=2\pi\int_0^1 e^{\frac{\big[\phi\big(2\log\frac1r\big)\big]^2}
{\theta_\beta(\phi)}}r\dd r\\
&=\pi\int_0^\infty
e^{\frac{\phi(t)^2}{\theta_\beta(\phi)}-t}\dd t.
\end{align*}
Finally, since $u_{\beta,\phi}$ is admissible in the definition of
$\mathrm{TM}_{\alpha_\beta,\beta}$, we obtain $J_{\alpha_\beta,\beta}(u_{\beta,\phi})\le\mathrm{TM}_{\alpha_\beta,\beta}$.
\end{proof}
\begin{lemma}\label{lem:limit} Let $\phi$ be an $n$-node polygonal profile on $[0,\infty) $. Then, we have 
\begin{equation}\label{eq:theta-limit}
\lim_{\beta\rightarrow 1^{-}}\theta_\beta(\phi)
:=\theta_1(\phi)
=\frac{y_1^2}{t_1}
e^{\int_{t_1}^{t_n}\frac{t\phi'(t)^2}{\phi(t)^2}\dd t},
\end{equation}
where $\theta_{\beta}(\phi)$, $0\le \beta<1$ is defined in \eqref{eq:theta}. In addition, if $a_j=(y_{j+1}-y_j)/(t_{j+1}-t_j)$ and $b_j=y_j-a_jt_j$ for $1\le j<n$, then
\begin{equation}\label{eq:theta-explicit}
\theta_1(\phi)=\frac{y_1y_n}{t_1}e^{\sum_{j=1}^{n-1}b_j
\big(\frac1{y_{j+1}}-\frac1{y_j}\big)}.
\end{equation}
\end{lemma}

\begin{proof}
Recall that $\phi(t)=(y_1/t_1)t$ on $[0,t_1]$. Thus, setting $E_1=y_1^2/t_1$ and  computing the first affine piece in \eqref{eq:normalization} exactly gives
\[
(1-\beta)\int_0^{t_1}\frac{y_1^2}{t_1^2}
\left(\frac{t_1^2}{y_1^2t}\right)^\beta\dd t
=E_1^{1-\beta}.
\]
Consequently,
\begin{equation}\label{eq:first-piece}
\|v_{\beta,\phi}\|_{w_\beta}^2=E_1^{1-\beta}+(1-\beta)R_\beta, \;\;\mbox{where}\;\;
R_\beta=\int_{t_1}^{t_n}\phi'(t)^2
\left(\frac{t}{\phi(t)^2}\right)^\beta\dd t.
\end{equation}
Noticing that $\phi\ge y_1>0$ on $[t_1, \infty)$, we can apply the dominated convergence to obtain
\[
R_\beta\rightarrow R_1
:=\int_{t_1}^{t_n}\frac{t\phi'(t)^2}{\phi(t)^2}\dd t,\;\; \mbox{as}\;\; \beta\rightarrow 1^{-}.
\]
Since 
\begin{align*}
  E_1^{1-\beta}=e^{(1-\beta)\log E_1}=1+(1-\beta)\log E_1+O((1-\beta)^2),\;\;\mbox{as}\;\; \beta\rightarrow 1^{-}
\end{align*}
we also have 
\begin{equation}\label{E1eR}
E_1^{1-\beta}+(1-\beta)R_\beta
=1+(1-\beta)(\log E_1+R_1)+o(1-\beta)
\end{equation}
as $\beta\rightarrow 1^{-}$.
Hence, by \eqref{eq:first-piece}, \eqref{E1eR}  and using that  $\log(1+s)=s+o(s)$, as $s\to 0$, we can write 
\begin{equation}\label{thetaR*}
\begin{aligned}
 \log\theta_\beta(\phi)&=\frac{1}
 {1-\beta}\log \|v_{\beta,\phi}\|_{w_\beta}^{2} \\
 &=\frac{1}{1-\beta}\log(E_1^{1-\beta}+(1-\beta)R_\beta)\\
 &=\frac{1}{1-\beta}\log\big(1+(1-\beta)(\log E_1+R_1)+o(1-\beta)\big)\\
 &=\log E_1+ R_1+o(1), \;\;\mbox{as}\;\; \beta\rightarrow 1^{-}.
\end{aligned}
\end{equation}
This proves \eqref{eq:theta-limit}. Now, since $\phi$ is an $n$-node polygonal profile, it admits the representation $\phi=b_j+a_jt$ on each interval $[t_j,t_{j+1}]$ for $1\le j<n$ and
\[
\frac{\dd}{\dd t}\left(\log\phi+\frac{b_j}{\phi}\right)
=\frac{t{\phi^{\prime}(t)}^2}{\phi^2(t)},\;\; \mbox{for}\;\;t\in [t_j, t_{j+1}].
\]
Thus, 
\[
\begin{aligned}
\int_{t_j}^{t_{j+1}}
\frac{t\,\phi'(t)^2}{\phi(t)^2}\,\mathrm dt
&=\left[
\log\phi(t)+\frac{b_j}{\phi(t)}
\right]_{t_j}^{t_{j+1}}\\
&=\log\frac{y_{j+1}}{y_j}
+b_j\left(\frac{1}{y_{j+1}}-\frac{1}{y_j}\right).
\end{aligned}
\]
By summing we get
\[
R_1=
\int_{t_1}^{t_n}
\frac{t\,\phi'(t)^2}{\phi(t)^2}\,\mathrm dt
=
\log\frac{y_n}{y_1}
+\sum_{j=1}^{n-1}
b_j\left(\frac{1}{y_{j+1}}-\frac{1}{y_j}\right).
\]
From \eqref{thetaR*}, we have $\theta_1=\frac{y_1^2}{t_1}e^{R_1}$. Thus, the above identity yields
 \eqref{eq:theta-explicit}.
\end{proof}

\begin{proposition}\label{prop:uniform-profile}
For every $\beta\in[0,1)$ and every $n$-node polygonal profile $\phi$ on $[0,\infty)$, we have 
\begin{equation}\label{eq:uniform-profile}
\mathrm{TM}_{\alpha_\beta,\beta}\ge
\pi \int_0^\infty e^{\frac{\phi(t)^2}{\theta_1(\phi)}-t}\dd t.
\end{equation}
\end{proposition}
\begin{proof}
Fix $\beta\in [0, 1)$ and choose $\beta'\in(\beta,1)$. The monotonicity in \eqref{mono-F} and Lemma~\ref{lem:test2} give
\[
\mathrm{TM}_{\alpha_\beta,\beta}\ge \mathrm{TM}_{\alpha_{\beta'},\beta'}
\ge\pi\int_0^\infty e^{\frac{\phi(t)^2}{\theta_{\beta'}(\phi)}-t}\dd t.
\]
Since $\phi\le y_n$, from Lemma~\ref{lem:limit} we have 
$$
 \frac{\phi(t)^2}{\theta_{\beta'}(\phi)}\to \frac{\phi(t)^2}{\theta_1(\phi)},\;\;\mbox{for all}\;\;t\ge 0\;\;\mbox{as}\;\;\beta^{\prime}\rightarrow 1^{-}
$$
and 
$$
0\le \frac{\phi(t)^2}{\theta_{\beta'}(\phi)}\le \frac{y^2_n}{\theta_{\beta^{\prime}}(\phi)}\to \frac{y^2_n}{\theta_1(\phi)}\;\;\mbox{as}\;\;\beta^{\prime}\rightarrow 1^{-},\;\; \mbox{for all}\;\; t\ge 0.
$$
Hence, the dominated convergence theorem implies 
\[
\mathrm{TM}_{\alpha_\beta,\beta}\ge \lim_{\beta^{\prime}\rightarrow 1^{-}}\mathrm{TM}_{\alpha_{\beta'},\beta'}
\ge\pi \lim_{\beta^{\prime}\rightarrow 1^{-}}\int_0^\infty e^{\frac{\phi(t)^2}{\theta_{\beta'}(\phi)}-t}\dd t=\pi\int_0^\infty e^{\frac{\phi(t)^2}{\theta_1(\phi)}-t}\dd t.
\]
\end{proof}

\begin{proposition}\label{prop:explicit}
For every $\beta\in[0,1)$, we have $\mathrm{TM}_{\alpha_\beta,\beta}>\frac{2267}{600}\pi>\pi(1+e)$.
\end{proposition}

\begin{proof}
Consider the three-node polygonal profile on $[0,\infty)$ (see Figure~\ref{fig:explicit-profile}) given by
\begin{equation}\label{eq:explicit-profile} 
\phi_\natural(t)=
\begin{cases}
\frac32t,&0\le t\le8,\\
t+4,&8\le t\le10,\\
9+\frac t2,&10\le t\le12,\\
15,&t\ge12.
\end{cases}
\end{equation}
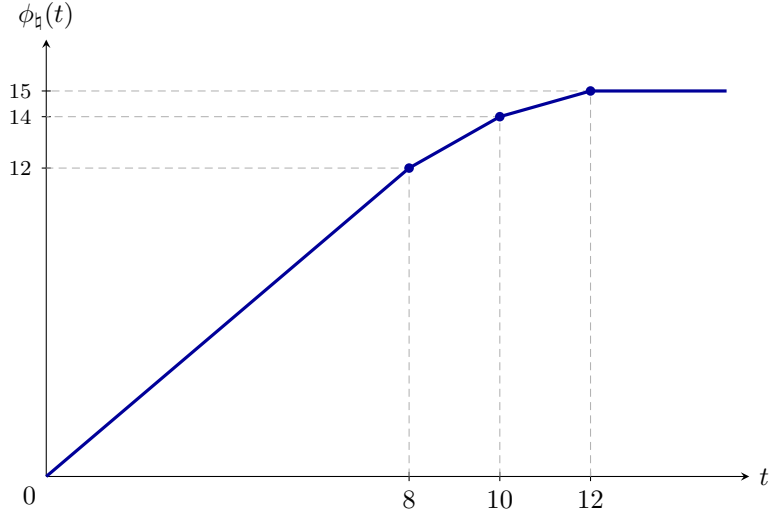
\begin{figure}[htbp]
\centering
\begin{tikzpicture}[x=0.60cm,y=0.34cm,>=stealth,
  every node/.style={font=\small}]
  \draw[->] (0,0) -- (15.5,0) node[right] {$t$};
  \draw[->] (0,0) -- (0,17) node[above] {$\phi_\natural(t)$};
  \node[below left] at (0,0) {$0$};
  \foreach \t in {8,10,12}
    \draw (\t,0.18) -- (\t,-0.18) node[below] {$\t$};
  \foreach \h in {12,14,15}
    \draw (0.10,\h) -- (-0.10,\h) node[left,font=\scriptsize] {$\h$};
  \draw[densely dashed,gray!60] (8,0) -- (8,12) -- (0,12);
  \draw[densely dashed,gray!60] (10,0) -- (10,14) -- (0,14);
  \draw[densely dashed,gray!60] (12,0) -- (12,15) -- (0,15);
  \draw[very thick,blue!60!black]
    (0,0) -- (8,12) -- (10,14) -- (12,15) -- (15,15);
  \fill[blue!60!black] (8,12) circle (1.8pt);
  \fill[blue!60!black] (10,14) circle (1.8pt);
  \fill[blue!60!black] (12,15) circle (1.8pt);
\end{tikzpicture}
\caption{The three-node profile in \eqref{eq:explicit-profile}, with positive nodes $(8,12)$, $(10,14)$, and $(12,15)$ and constant value $15$ for $t\ge12$.}
\label{fig:explicit-profile}
\end{figure}
By \eqref{eq:theta-explicit}, we have
\[
\theta_1(\phi_\natural)=\frac{45}{2}e^{-\frac{19}{210}}.
\]
Since
\[
e^{\frac{19}{210}}
>1+\frac{19}{210}
+\frac12\left(\frac{19}{210}\right)^2
=\frac{35}{32}+\frac{289}{352800}
>\frac{35}{32}
\]
we have
 $$
 \frac{1}{\theta_1(\phi_\natural)}>\frac7{144}.
 $$
From  Proposition~\ref{prop:uniform-profile} it follows that
\begin{equation}\label{eq:simple-profile-bound}
\frac{\mathrm{TM}_{\alpha_\beta,\beta}}{\pi}
\ge\int_0^\infty e^{\frac{\phi_\natural(t)^2}{\theta_1(\phi_\natural)}-t}\dd t
\ge \int_0^\infty e^{q(t)}\dd t,
\;\;\mbox{where}\;\; q(t)=\frac7{144}\phi_\natural(t)^2-t.
\end{equation}
We estimate this integral piece by piece, using tangent-line bounds for the convex function $q$ on $[0,8]$, as in the proof of \cite[Lemma~2.1]{HudsonLeckband02}, and Jensen's inequality on $[8,12]$.

On $[0,8]$ we have $q(t)=\frac7{64}t^2-t$, which is convex, with $q(0)=0$, $q'(0)=-1$, minimum $q(\frac{32}7)=-\frac{16}7$, $q(8)=-1$, and $q'(8)=\frac34$. Since a convex function lies above its tangent lines,
\[
q(t)\ge h(t)=\max\left\{-t,\,-\frac{16}7,\,\frac34t-7\right\},\;\;\mbox{for}\;\; 0\le t\le8.
\]
\begin{figure}[H]
\centering
\begin{tikzpicture}[x=1.30cm,y=0.88cm,>=stealth,
  every node/.style={font=\small},
  qcurve/.style={blue!70!black,very thick},
  envelope/.style={orange!85!black,line width=1.6pt},
  tangent/.style={gray!75,thick}]
  \foreach \x in {2.285714285714,4.571428571429,6.285714285714,8}
    \draw[gray!12] (\x,-4.5) -- (\x,0.4);
  \foreach \y in {-4,-3,-2.285714285714,-1}
    \draw[gray!12] (0,\y) -- (8.05,\y);
  \begin{scope}
    \clip (0,-4.5) rectangle (8.05,0.4);
    \path[fill=orange!10]
      plot[domain=0:8,samples=160] (\x,{7*\x*\x/64-\x})
      -- (8,-1) -- ({44/7},{-16/7}) -- ({16/7},{-16/7}) -- (0,0) -- cycle;
    \draw[tangent,dashed] (0,0) -- (8,-8);
    \draw[tangent,dashdotted] (0,{-16/7}) -- (8,{-16/7});
    \draw[tangent,densely dotted] (0,-7) -- (8,-1);
    \draw[qcurve] plot[domain=0:8,samples=160] (\x,{7*\x*\x/64-\x});
    \draw[envelope] (0,0) -- ({16/7},{-16/7})
      -- ({44/7},{-16/7}) -- (8,-1);
    \draw[orange!65!black,densely dashed] ({16/7},{-16/7}) -- ({16/7},0);
    \draw[orange!65!black,densely dashed] ({44/7},{-16/7}) -- ({44/7},0);
    \draw[blue!50!black,densely dotted] ({32/7},{-16/7}) -- ({32/7},0);
  \end{scope}
  \draw[->,gray!80] (-0.08,0) -- (8.35,0) node[right,black] {$t$};
  \draw[->,gray!80] (0,-4.55) -- (0,0.45);
  \node[above right] at (0,0.45) {$q(t),\ h(t)$};
  \node[below left] at (0,0) {$0$};
  \foreach \x in {2.285714285714,4.571428571429,6.285714285714,8}
    \draw[gray!65] (\x,-0.06) -- (\x,0.06);
  \node[below] at ({16/7},0) {$\frac{16}{7}$};
  \node[below] at ({32/7},0) {$\frac{32}{7}$};
  \node[below] at ({44/7},0) {$\frac{44}{7}$};
  \node[below] at (8,0) {$8$};
  \foreach \y in {-4,-3,-1}
    \draw[gray!65] (-0.06,\y) -- (0.06,\y) node[left,black,xshift=-2pt] {$\y$};
  \draw[gray!65] (-0.06,{-16/7}) -- (0.06,{-16/7});
  \node[left] at (-0.06,{-16/7}) {$-\frac{16}{7}$};
  \fill[blue!70!black] (0,0) circle (2pt);
  \fill[blue!70!black] ({32/7},{-16/7}) circle (2pt);
  \fill[blue!70!black] (8,-1) circle (2pt);
  \fill[orange!85!black] ({16/7},{-16/7}) circle (1.8pt);
  \fill[orange!85!black] ({44/7},{-16/7}) circle (1.8pt);
  \draw[qcurve] (0.05,-5.15) -- (0.48,-5.15);
  \node[right] at (0.58,-5.15) {$q(t)=\frac7{64}t^2-t$};
  \draw[tangent,dashed] (4.40,-5.15) -- (4.83,-5.15);
  \node[right] at (4.93,-5.15) {Tangent at $0$: $-t$};
  \draw[tangent,dashdotted] (0.05,-5.70) -- (0.48,-5.70);
  \node[right] at (0.58,-5.70) {Tangent at the minimum: $-\frac{16}{7}$};
  \draw[tangent,densely dotted] (4.40,-5.70) -- (4.83,-5.70);
  \node[right] at (4.93,-5.70) {Tangent at $8$: $\frac34t-7$};
  \draw[envelope] (0.05,-6.25) -- (0.48,-6.25);
  \node[right] at (0.58,-6.25) {$h(t)=\max\{-t,-\frac{16}{7},\frac34t-7\}$};
\end{tikzpicture}
\caption{The parabola $q$ (blue), its three tangent lines (gray), and their maximum $h$ (orange). The maximum changes branches at $\frac{16}{7}$ and $\frac{44}{7}$.}
\label{fig:q-tangents}
\end{figure}
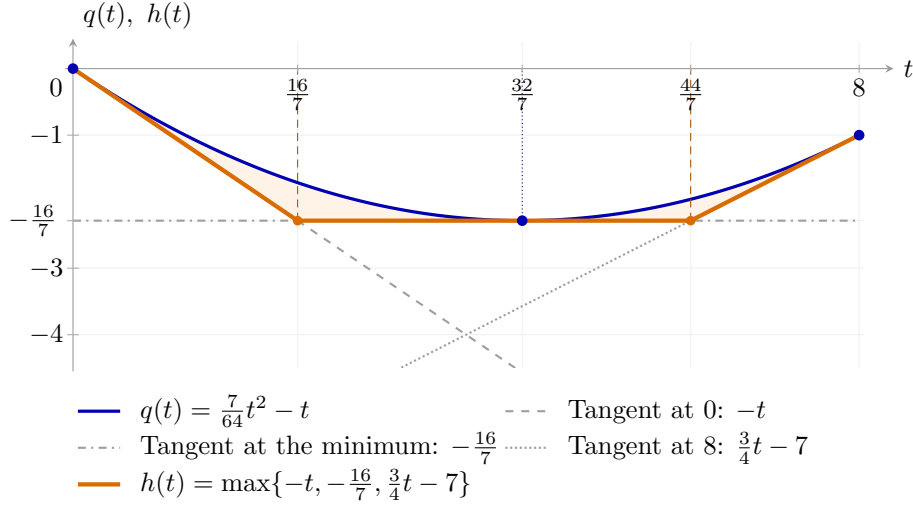
\noindent The maximum of these three affine functions changes branches at \(t=\frac{16}{7}\) and \(t=\frac{44}{7}\), so
\[
\int_0^8e^{q(t)}\dd t
\ge\int_0^{\frac{16}7}e^{-t}\dd t+4e^{-\frac{16}7}+\int_{\frac{44}7}^8e^{\frac34t-7}\dd t
=1+\frac43e^{-1}+\frac53e^{-\frac{16}7}.
\]
On $[8,12]$ we use the standard Jensen inequality
\begin{equation}\label{Jensen}
  \int_a^be^{q(t)}\dd t\ge(b-a)e^{\frac1{b-a}\int_a^bq(t)\dd t}.
\end{equation}
We have 
\[
\frac12\int_8^{10}q(t)\dd t=\frac7{144}\cdot\frac{14^3-12^3}{6}-9=-\frac{83}{108},
\]
and
\[
\frac12\int_{10}^{12}q(t)\dd t=\frac7{144}\cdot\frac{15^3-14^3}{3}-11=-\frac{335}{432},
\]
and both are larger than $-\frac79$. Hence, from \eqref{Jensen}
\begin{equation}\nonumber
\begin{aligned}
\int_8^{12}e^{q(t)}\dd t& =\int_8^{10}e^{q(t)}\dd t+\int_{10}^{12}e^{q(t)}\dd t\\
&\ge 2e^{\frac{1}{2}\int_{8}^{10}q(t)\dd t}+2e^{\frac{1}{2}\int_{10}^{12}q(t)\dd t}\\
&\ge 4e^{-\frac79}.
\end{aligned}
\end{equation}
Finally, $q(t)=\frac{7}{144}\cdot 225-t=\frac{175}{16}-t$ on $[12,\infty)$, so $\int_{12}^\infty e^{q(t)}\dd t=e^{-\frac{17}{16}}$. Consequently,
\[
\int_0^\infty e^{q(t)}\dd t
\ge1+\frac43e^{-1}+\frac53e^{-\frac{16}{7}}+4e^{-\frac79}+e^{-\frac{17}{16}}.
\]
Now,  we observe that, for $0\le z<1$,
\begin{equation}\label{log-inequality}
  \log\frac{1+z}{1-z}
=2\int_0^z\frac{1}{1-s^2}\dd s\ge 2z.
\end{equation}
Taking $z={x}/{8}$ yields 
\begin{equation}\nonumber
  e^{-x}\ge \left(\frac{8-x}{8+x}\right)^4,
\;\;\mbox{for}\;\; 0<x<8.
\end{equation}
In particular,
\[
e^{-1}\ge \left(\frac79\right)^4>\frac9{25}\;\;\mbox{and}\;\;
e^{-\frac{16}{7}}\ge \left(\frac59\right)^4>\frac{19}{200},
\]
and also
\[
e^{-\frac79}\ge \left(\frac{65}{79}\right)^4>\frac9{20}\;\;\mbox{and}\;\;
e^{-\frac{17}{16}}\ge \left(\frac{111}{145}\right)^4>\frac{17}{50}.
\]
It follows that
\[
\int_0^\infty e^{q(t)}\dd t
>1+\frac43\cdot\frac9{25}+\frac53\cdot\frac{19}{200}+4\cdot\frac9{20}+\frac{17}{50}
=\frac{2267}{600}>1+e,
\]
since we have $e<\frac{11}{4}<\frac{1667}{600}$. Together with \eqref{eq:simple-profile-bound}, this completes the proof.
\end{proof}
\begin{proof}[Proof of Theorem~\ref{thm:attainment}]
  First, as pointed out in \cite{PaivaOliveira26}, from  the results of B.~Opic and A.~Kufner \cite[Theorems~7.4 and~7.13]{Opic}, we have the compact embeddings
\begin{equation}\label{emb-compact}
  H^{1}_{0,\operatorname{rad}}(B, w_{\beta})\hookrightarrow L^{q}(B),\;\; \mbox{for}\;\; 1<q<\infty.
\end{equation}
Fix $\beta\in[0,1)$ and let $(u_k)\subset H^{1}_{0,\mathrm{rad}}(B, w_{\beta})$ be a maximizing sequence for $\mathrm{TM}_{\alpha_{\beta},\beta}$. Since $t\mapsto J_{\alpha_{\beta},\beta}(tu_k)$,  $t\ge 0$ is a nondecreasing function, we can assume $\|u_k\|_{w_{\beta}}=1$. In addition, up to a subsequence, we can assume that $u_k\weakto u$ in $H^1_{0,\mathrm{rad}}(B,w_\beta)$ and $u_k(x)\to u(x)$ almost everywhere on $B$. By the weak lower semicontinuity of the norm we also have $\|u\|_{w_\beta}\le1$. Proposition~\ref{prop:explicit} and \eqref{concentration-B} give
\begin{equation}\label{TM>frente}
  \mathrm{TM}_{\alpha_\beta,\beta}>\pi(1+e)\ge\Sigma_\beta.
\end{equation}
If $u=0$, the alternative in \cite[Lemma~5]{Roy16} gives either $J_{\alpha_\beta,\beta}(u_k)\to\pi$ or $(u_k)$ is a normalized concentrating sequence. Both possibilities contradict the strict bound \eqref{TM>frente}. Thus $u\ne0$. By \cite[Lemma~2.2]{Nguyen19}, there exists $p>1$ such that 

\begin{equation}
\limsup_k\int_{B}e^{p\alpha_\beta|u_k|^{\frac{2}{1-\beta}}}\dd x<\infty.
\end{equation} 
 So, Vitali's theorem yields
$
J_{\alpha_\beta,\beta}(u)=\mathrm{TM}_{\alpha_\beta,\beta}.
$
Thus, $u$  is the desired maximizer.
\end{proof}

\begin{remark}\label{extremal-normalized}
  The maximizer $u$ ensured by Theorem~\ref{thm:attainment} satisfies $\|u\|_{w_{\beta}}=1$. Indeed, if $0<\|u\|_{w_\beta}<1$, we must have 
$
J_{\alpha_\beta,\beta}\big({u}/{\|u\|_{w_\beta}}\big)
>J_{\alpha_\beta,\beta}(u)=\mathrm{TM}_{\alpha_\beta,\beta},
$
which is impossible. 
\end{remark}
%======================================================================
\section{Local maximizers beyond the critical threshold}\label{sec:supercritical}
%======================================================================

In this section, we prove Theorem~\ref{thm:supercritical} by extending \cite[Lemmas~3--5]{PaivaOliveira26} to the full range $0\le\beta<1$. In \cite{PaivaOliveira26}, these lemmas are stated for small $\beta$, and this restriction is used only through two facts: the strict bound $\mathrm{TM}_{\alpha_\beta,\beta}>\pi(1+e)$ \cite[(3.1)]{PaivaOliveira26} and the attainment of $\mathrm{TM}_{\alpha_\beta,\beta}$, which gives \cite[(3.2)]{PaivaOliveira26}. Both facts now hold for every $\beta\in[0,1)$, by Proposition~\ref{prop:explicit} and Theorem~\ref{thm:attainment}. The remaining ingredients, namely \cite[Lemma~1]{PaivaOliveira26} ($J_{\alpha,\beta}$ is of class $C^1$ and $J'_{\alpha,\beta}$ is locally bounded) and \cite[Lemma~2]{PaivaOliveira26} (weak solutions of the Euler--Lagrange equation, in particular the elements of $K_{\alpha_\beta,\beta}$ below, are bounded), are stated there for every $\beta\in[0,1)$. Hence, the strict bound $\mathrm{TM}_{\alpha_\beta,\beta}>\pi(1+e)\ge\Sigma_\beta$ is the only new ingredient, and it enters through the following compactness property.

\begin{lemma}\label{lem:extension3}
Let $\beta\in[0,1)$. Every sequence $(u_k)\subset\mathcal S_\beta$ such that $J_{\alpha_\beta,\beta}(u_k)\to\mathrm{TM}_{\alpha_\beta,\beta}$ has a subsequence that converges in $H^1_{0,\mathrm{rad}}(B,w_\beta)$ to some element of
\[
K_{\alpha_{\beta},\beta}=
\{u\in\mathcal S_\beta:
J_{\alpha_\beta,\beta}(u)=\mathrm{TM}_{\alpha_\beta,\beta}\}.
\]
In particular, $K_{\alpha_{\beta},\beta}$ is nonempty and compact in $H^1_{0,\mathrm{rad}}(B,w_\beta)$.
\end{lemma}

\begin{proof}
Up to a subsequence, we can assume that $u_k\weakto u$ in $H^1_{0,\mathrm{rad}}(B,w_\beta)$. Since $\mathrm{TM}_{\alpha_\beta,\beta}>\pi(1+e)\ge\Sigma_\beta$ by Proposition~\ref{prop:explicit} and \eqref{concentration-B}, the argument in the proof of Theorem~\ref{thm:attainment} gives $u\ne0$ and $J_{\alpha_\beta,\beta}(u)=\mathrm{TM}_{\alpha_\beta,\beta}$, and Remark~\ref{extremal-normalized} gives $\|u\|_{w_\beta}=1$. Thus $u\in K_{\alpha_{\beta},\beta}$ and, since $\|u_k\|_{w_\beta}=1$ for all $k$, the weak convergence implies that $u_k\to u$ strongly in $H^1_{0,\mathrm{rad}}(B,w_\beta)$. Finally, $K_{\alpha_{\beta},\beta}\ne\varnothing$ by Theorem~\ref{thm:attainment} and Remark~\ref{extremal-normalized}, and $K_{\alpha_{\beta},\beta}$ is compact because every sequence in $K_{\alpha_{\beta},\beta}$ is a maximizing sequence.
\end{proof}
\noindent For each $\eps>0$, let
\[
N_\eps=\{u\in\mathcal S_\beta\, :\,
\dist\,(u,K_{\alpha_{\beta},\beta})<\eps\},\;\; \mbox{where}\;\; \dist\,(u,K_{\alpha_{\beta},\beta})
=\inf_{v\in K_{\alpha_{\beta},\beta}}\|u-v\|_{w_\beta}.
\]
Since $K_{\alpha_{\beta},\beta}$ is compact, $(N_\eps)_{\eps>0}$ is a neighborhood basis for $K_{\alpha_{\beta},\beta}$ in $\mathcal S_\beta$. Moreover, $K_{\alpha_{\beta},\beta}\subset N_\eps$ and Theorem~\ref{thm:attainment} give
\begin{equation}\label{eq:critical-neighborhood-level}
\mathrm{TM}_{\alpha_\beta,\beta}
=\sup_{u\in N_\eps}J_{\alpha_\beta,\beta}(u),
\;\;\mbox{for every}\;\;\eps>0.
\end{equation}
With Lemma~\ref{lem:extension3} and \eqref{eq:critical-neighborhood-level} in place of \cite[Lemma~3 and (3.2)]{PaivaOliveira26}, and with Proposition~\ref{prop:explicit} in place of \cite[(3.1)]{PaivaOliveira26}, the proofs of \cite[Lemmas~4 and~5]{PaivaOliveira26} apply without change and give the following result.

\begin{lemma}\label{lem:extension5}
For every $\beta\in[0,1)$, there exist $\eps>0$ and $\alpha^\sharp>\alpha_\beta$ such that, for every $\alpha\in[\alpha_\beta,\alpha^\sharp)$,
\begin{equation}\label{eq:perturbed-gap}
\sup_{u\in N_{2\eps}\setminus N_\eps}J_{\alpha,\beta}(u)
<A_{\alpha,\beta}^{\sharp}:=\sup_{u\in N_\eps}J_{\alpha,\beta}(u)<\infty.
\end{equation}
Moreover, $A_{\alpha,\beta}^\sharp$ is attained in $N_\eps$, and the set
\[
K_{\alpha,\beta}=\{u\in N_\eps:J_{\alpha,\beta}(u)=A_{\alpha,\beta}^\sharp\}
\]
is compact in $H^1_{0,\mathrm{rad}}(B,w_\beta)$.
\end{lemma}

\begin{proof}[Proof of Theorem~\ref{thm:supercritical}]
In the subcritical case $0\le \alpha<\alpha_\beta$, the attainability follows from compact embedding \eqref{emb-compact} and the Vitali convergence theorem, and, as in Remark~\ref{extremal-normalized}, the maximizer belongs to $\mathcal S_\beta$ (for $\alpha=0$, $J_{0,\beta}$ is constant). Theorem~\ref{thm:attainment} ensures a maximizer in the critical case $\alpha=\alpha_\beta$, which belongs to $\mathcal S_\beta$ by Remark~\ref{extremal-normalized}. Finally, for $\alpha\in(\alpha_\beta,\alpha^\sharp)$, Lemma~\ref{lem:extension5} gives a maximizer in $N_\eps$. Since $N_\eps$ is open in $\mathcal S_\beta$, this is a local maximizer on $\mathcal S_\beta$. 
\end{proof}

%======================================================================
\section{A lower bound for the classical supremum}\label{sec:rational-classical}
%======================================================================

In this section, we first give a profile estimate for every $\beta\in[0,1)$ and then apply it with $\beta=0$ to estimate the classical Trudinger--Moser supremum.
Let  $\phi$ be an $n$-node polygonal profile, with $\phi(0)=0$, nodes $t_j>0$ and heights $y_j=\phi(t_j)$, $j=1, \dots, n$. Then, we set 
\begin{equation}\label{eq:jensen-profile-mean}
\mu_j=
\frac{y_{j-1}^2+y_{j-1}y_j+y_j^2}{3\theta_{\beta}(\phi)}
-\frac{t_{j-1}+t_j}{2},
\end{equation}
with $\theta_{\beta}(\phi)$ defined in \eqref{eq:theta}. With this notation, we have the following:
\begin{lemma}\label{lem:jensen-rational}
Let $\beta\in[0,1)$ and let $\phi$ be an $n$-node polygonal profile, and let $\mu_j$ be as in \eqref{eq:jensen-profile-mean}.
Then
\begin{equation}\label{eq:jensen-profile-bound}
\frac{\mathrm{TM}_{\alpha_{\beta},\beta}}{\pi}\ge \int_0^\infty e^{\frac{\phi(t)^2}{\theta_{\beta}(\phi)}-t}\dd t
\ge\sum_{j=1}^n (t_j-t_{j-1})e^{\mu_j}+e^{\frac{y_n^2}{\theta_{\beta}(\phi)}-t_n}.
\end{equation}
Moreover, $-t_n\le\mu_j\le0$ for $1\le j\le n$, and $-t_n\le y_n^2/\theta_{\beta}(\phi)-t_n\le0$.
\end{lemma}
\begin{proof}
Fix $j\in\{1,\dots,n\}$. Since $\phi$ is affine on $[t_{j-1},t_j]$ with endpoint values $y_{j-1}$ and $y_j$, we have 
$$
\phi(t)=y_{j-1}+\frac{y_j-y_{j-1}}{t_j-t_{j-1}}(t-t_{j-1}), \;\;\mbox{for all}\;\; t\in [t_{j-1}, t_j]. 
$$
Hence, the change of variables $t=t_{j-1}+s(t_j-t_{j-1})$, $s\in[0,1]$, gives
\[
\begin{aligned}
\frac{1}{t_j-t_{j-1}}\int_{t_{j-1}}^{t_j}\phi(t)^2\dd t
&=\int_0^1\big((1-s)y_{j-1}+sy_j\big)^2\dd s
=\frac{y_{j-1}^2+y_{j-1}y_j+y_j^2}{3}.
\end{aligned}
\]
Let $q(t)=\phi(t)^2/\theta_{\beta}(\phi)-t$. Thus 
\begin{equation}\label{mu-media}
 \begin{aligned}
 \frac{1}{t_j-t_{j-1}}\int_{t_{j-1}}^{t_j}q(t)\dd t=\frac{y_{j-1}^2+y_{j-1}y_j+y_j^2}{3\theta_{\beta}(\phi)}-\frac{t_j+t_{j-1}}{2}=\mu_j.
\end{aligned} 
\end{equation}
Thus, Jensen's inequality \eqref{Jensen} gives
\begin{equation}\label{sum-part}
\int_{t_{j-1}}^{t_j}e^{q(t)}\dd t\ge (t_j-t_{j-1})e^{\mu_j}.
\end{equation}
Finally, on $[t_n, \infty)$ we have $\phi=y_n$, then 
\begin{equation}\label{tail-part}
  \int_{t_n}^\infty e^{q(t)}\dd t=e^{\frac{y_n^2}{\theta_{\beta}(\phi)}-t_n}.
\end{equation}
Now, by combining \eqref{sum-part} with \eqref{tail-part} and using \eqref{eq:direct-lower},  we obtain \eqref{eq:jensen-profile-bound}.  By H\"{o}lder's inequality and \eqref{change-norm}, for $t>0$,
\begin{equation}
\begin{aligned}
\phi(t)^{1-\beta}
&=(1-\beta)\int_0^t\phi'(s)\phi(s)^{-\beta}\dd s\\
&\le\left((1-\beta)\int_0^t\phi'(s)^2
\left(\frac{s}{\phi(s)^2}\right)^\beta\dd s\right)^{\frac{1}{2}}
t^{\frac{1-\beta}{2}}\\
&\le{\theta_{\beta}(\phi)}^{\frac{1-\beta}{2}}t^{\frac{1-\beta}{2}}.
\end{aligned}
\end{equation}
Thus $\phi(t)^2\le\theta_{\beta}(\phi) t$. Hence $-t_n\le q(t)\le0$ on $[0,t_n]$. Thus, by integrating on $[t_{j-1}, t_j]$, the identity \eqref{mu-media} yields $-t_n\le \mu_j\le 0$ and also $-t_n\le q(t_n)=y_n^2/\theta_{\beta}(\phi)-t_n\le 0$.
\end{proof}

\begin{remark}\label{rem:jensen-uniform}
Letting $\beta\rightarrow 1^{-}$ in $\phi(t)^2\le\theta_{\beta}(\phi)t$ and using Lemma~\ref{lem:limit}, we obtain $\phi(t)^2\le\theta_1(\phi)t$ for $t\ge0$. Hence, with Proposition~\ref{prop:uniform-profile} in place of \eqref{eq:direct-lower}, the proof of Lemma~\ref{lem:jensen-rational} shows that its conclusions remain valid with $\theta_1(\phi)$ in place of $\theta_{\beta}(\phi)$ in \eqref{eq:jensen-profile-mean} and \eqref{eq:jensen-profile-bound}, for every $\beta\in[0,1)$. In particular, this gives lower bounds for $\mathrm{TM}_{\alpha_\beta,\beta}$ that are uniform in $\beta\in[0,1)$.
\end{remark}

\begin{proof}[Proof of Theorem~\ref{thm:classical-bounds}]
By Schwarz symmetrization, we have
$\mathrm{TM}_{4\pi,0}=\mathrm{TM}_{4\pi}$. Let $\phi$ be an $n$-node polygonal profile  with nodes $t_j=hj$ with $h>0$, for $j=1,\cdots, n$ and heights $y_j$, and set
$y=(y_0,\ldots,y_n)$. From \eqref{change-norm} and \eqref{eq:theta} with $\beta=0$, we have 
\[
  \theta_0(\phi)=\int_{0}^{\infty}|\phi^{\prime}(t)|^2\dd t=\frac{1}{h}\sum_{j=1}^n(y_j-y_{j-1})^2.
\]
In addition, the inequality \eqref{eq:jensen-profile-bound} with $\beta=0$ gives
\begin{equation}\label{eq:classical-jensen-bound}
\mathrm{TM}_{4\pi}\ge\pi J_n(h,y),
\end{equation}
where
\begin{equation}\label{eq:classical-jensen-functional}
\begin{aligned}
J_n(h,y)=&h\sum_{j=1}^n
\exp \left[h\left(
\frac{y_{j-1}^2+y_{j-1}y_j+y_j^2}{3E(y)}
-j+\frac{1}{2}
\right)\right]+\exp\left[h\left(\frac{y_n^2}{E(y)}-n\right)\right]
\end{aligned}
\end{equation}
in which  
\[
E(y)=h\theta_0(\phi)
=\sum_{j=1}^n(y_j-y_{j-1})^2.
\]
Note that $E(cy)=c^2E(y)$ and therefore $J_n(h,cy)=J_n(h,y)$ for every
$c>0$. We may thus impose $y_n=1$. Together with $y_0=0$, this leaves the
$n$ parameters $h,y_1,\ldots,y_{n-1}$. To prove the claimed lower bound,
it suffices to give one admissible choice of these parameters. Take $n=30$, $h=13/30$ and $y_j=Y_j/3000$, where $Y_0=0$ and the integers $Y_j$ are listed in Table~\ref{tab:thirty-node-heights} in Appendix~\ref{app:thirty-node-data}.
These heights are strictly increasing and $y_{30}=1$. The corresponding
profile is constant after $t_{30}=13$, and
\begin{equation}\label{eq:thirty-node-energy}
\sum_{j=1}^{30}(Y_j-Y_{j-1})^2=369200,\;\;
 E(y)=\frac{923}{22500}\;\;\mbox{and}\;\;
\theta_0(\phi)=\frac{71}{750}.
\end{equation}
Substituting into \eqref{eq:classical-jensen-functional}, set
\begin{equation}\label{eq:classical-exponent-values}
\begin{aligned}
q_j&=\frac{Y_{j-1}^2+Y_{j-1}Y_j+Y_j^2-553800(2j-1)}{2556000},\;\;\mbox{for}\;\;  1\le j\le30,\\
r&=\frac1{\theta_0(\phi)}-13=-\frac{173}{71}.
\end{aligned}
\end{equation}
The exact values of all $q_j$ are given by $q_j=-m_j/2556000$, with the
integers $m_j$ listed in Table~\ref{tab:thirty-node-exponents}
in Appendix~\ref{app:thirty-node-data}.
Thus the test value is the explicit sum
\begin{equation}\label{eq:classical-profile-evaluation}
J_{30}\left(\frac{13}{30},y\right)
=\frac{13}{30}\sum_{j=1}^{30}e^{q_j}+e^{-\frac{173}{71}}.
\end{equation}
We verify the required strict lower bound by the elementary inequality
$e^s\ge1+s+s^2/2$, for $s\ge0$, which gives
\begin{equation}\label{eq:classical-quadratic-bound}
e^q\ge e^{-\frac{k}{8}}\left[1+q+\frac k8+
\frac12\left(q+\frac k8\right)^2\right],\;\;\mbox{for}\;\;
 q\ge-\frac k8.
\end{equation}
Taking $z=1/128$ in \eqref{log-inequality}, we have $\log\frac{129}{127}\ge\frac{1}{64}$, and hence
\[
e^{-\frac18}\ge
\left(\frac{127}{129}\right)^8>\frac{368}{417}.
\]
Set
\begin{equation}\label{eq:classical-base-recurrence}
B_0=10^6,\qquad
B_{k+1}=\left\lfloor\frac{368B_k}{417}\right\rfloor,
\;\;\mbox{for}\;\; 0\le k<20,
\end{equation}
where $\lfloor x\rfloor=\max\{m\in\mathbb Z:m\le x\}$. It follows that  
\[
10^{-6}B_k\le\left(\frac{368}{417}\right)^k<e^{-\frac{k}{8}},
\;\;\mbox{for}\;\; 1\le k\le20.
\]
For each $j\in\{1,\ldots,30\}$, we choose
$k=\lceil -8q_j\rceil$, where
$\lceil x\rceil=\min\{m\in\mathbb Z:m\ge x\}$. In particular,  $0\le q_j+k/8<1/8$.
Grouping the indices with the same $k$, we partition $\{1,\ldots,30\}$
into the sets $I_k$ listed in Table~\ref{tab:classical-index-groups}.
For each listed $k\ne20$, put
\[
\begin{aligned}
C_k&=\frac{13}{30}\sum_{j\in I_k}
\left[1+q_j+\frac k8+\frac12\left(q_j+\frac k8\right)^2\right],\\
C_{20}&=1+r+\frac{20}{8}+\frac12\left(r+\frac{20}{8}\right)^2
=\frac{42965}{40328}.
\end{aligned}
\]
For the tail, we have $r+20/8=9/142\in[0,1/8)$.
Using $q_j=-m_j/2556000$, with the integers $m_j$ listed in
Table~\ref{tab:thirty-node-exponents}, we obtain $C_k\ge A_k\cdot10^{-5}$,
with the integers $A_k$ in Table~\ref{tab:classical-index-groups}. The last column is obtained from
\eqref{eq:classical-base-recurrence}.
\begin{table}[H]
\centering
\[
\begin{array}{c|l|r|r}
k&I_k&A_k&B_k\\\hline
2&\{1\}&45176&778795\\
4&\{22,23\}&89997&606521\\
5&\{2,21,24\}&139644&535251\\
6&\{20,25\}&93685&472355\\
7&\{19\}&45381&416850\\
8&\{3,26\}&94630&367867\\
9&\{18\}&48386&324640\\
10&\{4,17,27\}&136930&286492\\
12&\{5,16,28\}&137147&223118\\
13&\{15\}&46082&196900\\
14&\{6,14\}&93124&173763\\
15&\{7,13,29\}&136343&153344\\
16&\{8,9,10,11,12\}&228165&135325\\
18&\{30\}&43722&105390\\
20&\mbox{tail}&106538&82077
\end{array}
\]
\caption{Index groups $I_k$ and integers $A_k$, $B_k$ used in the lower bound for $J_{30}$. The last row corresponds to the constant tail of the profile.}
\label{tab:classical-index-groups}
\end{table}
All entries in the coefficient column require only rational arithmetic.
For example,
\[
\begin{aligned}
q_1&=-\frac{133121}{639000},\qquad q_1+\frac14=\frac{26629}{639000},\\
C_2&=\frac{13}{30}\left[1+\frac{26629}{639000}+
\frac12\left(\frac{26629}{639000}\right)^2\right]
>45176\cdot10^{-5}.
\end{aligned}
\]
Since the listed sets $I_k$, $k\ne20$, partition $\{1,\ldots,30\}$ and
all $C_k$ are positive, \eqref{eq:classical-quadratic-bound} gives
\[
J_{30}\left(\frac{13}{30},y\right)
\ge\sum_k e^{-\frac{k}{8}}C_k
>10^{-11}\sum_k A_kB_k
=4.38421793217,
\]
where both sums run over the rows of Table~\ref{tab:classical-index-groups}. Consequently,
\begin{equation}\label{eq:classical-direct-certificate}
\frac{\mathrm{TM}_{4\pi,0}}{\pi}
\ge J_{30}\left(\frac{13}{30},y\right)
>4.38421793217.
\end{equation}
\end{proof}

%======================================================================
\appendix
\section{Numerical data and computations}\label{app:numerical}
%======================================================================

\subsection{Heights and exponents of the thirty-node profile}\label{app:thirty-node-data}

The choice of $h=13/30$ and of the heights listed below is based on approximate
parameter values obtained by numerically maximizing $J_{30}(h,y)$ under
the normalization $y_0=0$ and $y_{30}=1$. These parameters were then
replaced by nearby rational values; the proof of Theorem~\ref{thm:classical-bounds} uses only the
resulting explicit profile, with nodes $t_j=13j/30$ and heights $y_j=Y_j/3000$, where the
integers $Y_j$ are given in Table~\ref{tab:thirty-node-heights}.

\begin{table}[H]
\centering
\begin{tabular}{c|r@{\qquad}c|r@{\qquad}c|r}
$j$&$Y_j$&$j$&$Y_j$&$j$&$Y_j$\\\hline
1&146&11&1549&21&2709\\
2&291&12&1682&22&2783\\
3&436&13&1813&23&2844\\
4&579&14&1942&24&2892\\
5&721&15&2068&25&2928\\
6&862&16&2191&26&2954\\
7&1002&17&2309&27&2973\\
8&1141&18&2422&28&2986\\
9&1278&19&2527&29&2994\\
10&1414&20&2623&30&3000
\end{tabular}
\caption{Integers defining the heights $y_j=Y_j/3000$ of the thirty-node profile.}
\label{tab:thirty-node-heights}
\end{table}
\noindent For these heights, \eqref{eq:classical-exponent-values} gives
\[
\begin{aligned}
q_j=-\frac{m_j}{2556000},\;\;\mbox{where}\;\;
m_j=553800(2j-1)-Y_{j-1}^2-Y_{j-1}Y_j-Y_j^2,
\;\;\mbox{for}\;\; 1\le j\le 30.
\end{aligned}
\]
Table~\ref{tab:thirty-node-exponents} lists the integers $m_j$. For instance, $$
m_2=3\cdot553800-146^2-146\cdot291-291^2=1512917.
$$
\begin{table}[H]
\centering
\begin{tabular}{c|r@{\qquad\qquad}c|r}
$j$&$m_j$&$j$&$m_j$\\\hline
1&532484&16&3559707\\
2&1512917&17&3084419\\
3&2367347&18&2593037\\
4&3098819&19&2118393\\
5&3711659&20&1704021\\
6&4207413&21&1381283\\
7&4588628&22&1190483\\
8&4857833&23&1172723\\
9&5021237&24&1351752\\
10&5082428&25&1731576\\
11&5040717&26&2295188\\
12&4903457&27&3004313\\
13&4679441&28&3826697\\
14&4373421&29&4746284\\
15&3996156&30&5728164
\end{tabular}
\caption{Integers defining the exact exponents $q_j=-m_j/2556000$ of the thirty-node profile.}
\label{tab:thirty-node-exponents}
\end{table}

\subsection{Numerical experiments with finer profiles}\label{app:numerical-classical}

Table~\ref{tab:numerical-classical} lists the values for
$\mathrm{TM}_{4\pi}/\pi$ from the numerical experiments mentioned in
Remark~\ref{rem:numerical-classical}. The aim is to illustrate how closely the $\mathrm{TM}_{4\pi}$ supremum can be approximated using an $n$-node polygonal profile.
They were obtained by numerically optimizing the spacing $h$ and the
heights $y_j$ in $J_n(h,y)$ (see \eqref{eq:classical-jensen-functional}) for
polygonal profiles with $n$ uniformly spaced nodes $t_j=jh$, normalized by
$y_0=0$ and $y_n=1$. The thirty-node row uses a more precise rational
choice of parameters than the profile in Appendix~\ref{app:thirty-node-data}.
For each $n$, the optimized heights were rounded to multiples of $10^{-12}$,
and the last column is a strict lower bound for $J_n(h,y)$ at the resulting
rational profile, obtained by truncating its value to nine decimal places.
This lower bound is certified in exact rational arithmetic by means of the
Taylor polynomials of degrees $59$ and $60$ of the exponential function, which
bound $e^x$ from below and from above, respectively, for $x\le0$ (all exponents
in $J_n(h,y)$ are nonpositive by Lemma~\ref{lem:jensen-rational}).
Thus these values are computer-assisted: each one is a strict lower bound for
$J_n(h,y)$, and hence for $\mathrm{TM}_{4\pi}/\pi$ by
\eqref{eq:classical-jensen-bound}, provided that the supplementary code is
correct. They are not used in any proof in this paper.
The rows with $n\ge51200$ lie above the lower bound $4.40273202$ obtained in
the proof of \cite[Proposition~7.7]{LiSun26}.

\begin{table}[H]
\centering
\renewcommand{\arraystretch}{1.5}
\begin{tabular}{rccc}
Nodes $n$ & Terminal point $T=nh$ & Spacing $h$ & Numerical lower bound\\\hline
$30$ & $\frac{1667}{128}$ & $\frac{1667}{3840}$ & $4.384870657$\\
$1600$ & $\frac{4391}{256}$ & $\frac{4391}{409600}$ & $4.402721141$\\
$3200$ & $\frac{1141}{64}$ & $\frac{1141}{204800}$ & $4.402729108$\\
$6400$ & $\frac{4737}{256}$ & $\frac{4737}{1638400}$ & $4.402731258$\\
$12800$ & $\frac{2455}{128}$ & $\frac{491}{327680}$ & $4.402731836$\\
$25600$ & $\frac{5083}{256}$ & $\frac{5083}{6553600}$ & $4.402731991$\\
$51200$ & $\frac{657}{32}$ & $\frac{657}{1638400}$ & $4.402732033$\\
$102400$ & $\frac{2715}{128}$ & $\frac{543}{2621440}$ & $4.402732044$\\
\end{tabular}
\caption[Numerical lower bounds for $\mathrm{TM}_{4\pi}/\pi$]{Numerical lower bounds for $\mathrm{TM}_{4\pi}/\pi$ obtained
experimentally by optimizing the spacing and heights in $J_n(h,y)$. Each row
is confirmed by the supplementary files \texttt{profile\_}$n$\texttt{.json}
(rational profile), \texttt{exponents\_}$n$\texttt{.json} (exact exponents),
\texttt{value\_}$n$\texttt{.json} (high-precision value) and
\texttt{certificate\_}$n$\texttt{.json} (exact rational lower bound); the
supplementary script \texttt{evaluate.rb} reproduces all rows. For comparison,
$4.40273202\le\mathrm{TM}_{4\pi}/\pi\le4.40273208$ (see the proof of
\cite[Proposition~7.7]{LiSun26}).}
\label{tab:numerical-classical}
\end{table}

%======================================================================
\section*{Statement on the use of AI}

During the preparation of this manuscript, the author used the AI tools ChatGPT and Codex (OpenAI) and Claude Code (Anthropic). The use of these tools was restricted to: (i) the organization and wording of the text; (ii) the verification of mathematical arguments and computations; and (iii) the writing of the code for the numerical experiments described in Remark~\ref{rem:numerical-classical}, whose results are listed in Appendix~\ref{app:numerical-classical}. These numerical experiments are not used in the proof of any result of the paper. The author takes full responsibility for the mathematical content, computations, references, and final text.

\section*{Funding}
J.\,F.~de Oliveira acknowledges partial support from the National Council for Scientific and Technological Development (CNPq, Brazil) through grant 303443/2025-1.
%======================================================================

\end{document}